\documentclass[a4paper]{article}
\usepackage[all]{xy}\usepackage[latin1]{inputenc}        
\usepackage[dvips]{graphics,graphicx}
\usepackage{amsfonts,amssymb,amsmath,color,mathrsfs, amstext}
\usepackage{amsbsy, amsopn, amscd, amsxtra, amsthm,authblk}
\usepackage{enumerate,algorithmic,algorithm}
\usepackage{upref}
\usepackage{geometry}
\usepackage[displaymath]{lineno}
\usepackage{float}
\usepackage{yhmath}
\usepackage{booktabs}
\usepackage{subcaption}
\usepackage{multirow}
\usepackage{makecell}
\usepackage[colorlinks,
            linkcolor=red,
            anchorcolor=red,
            citecolor=red
            ]{hyperref}

\numberwithin{equation}{section}

\def\e{\epsilon}

\def\R{\mathbb{R}}
\def\cL{\mathcal{L}}
\def\cD{\mathcal{D}}
\def\cP{\mathcal{P}}
\def\cU{\mathcal{U}}
\def\cA{\mathcal{A}}
\def\cS{\mathcal{S}}

\def\cX{\mathcal{X}}
\def\div{\mathrm{div}}
\def\loc{\mathrm{loc}}
\def\tr{\mathrm{tr}}
\def\l{\langle}
\def\r{\rangle}

\newtheorem{theorem}{Theorem}[section]
\newtheorem{lemma}{Lemma}[section]
\newtheorem{proposition}{Proposition}[section]
\newtheorem{remark}{Remark}[section]
\newtheorem{corollary}{Corollary}[section]
\newtheorem{definition}{Definition}[section]

\makeatletter
\@mparswitchfalse
\makeatother

\title{Existence of weak solutions to $p$-Navier-Stokes equations}
\author[a]{Yuanyuan Feng\thanks{
E-mail: yyfeng@math.ecnu.edu.cn}}
\author[b]{Lei Li\thanks{E-mail: leili2010@sjtu.edu.cn}}
\author[c]{Jian-Guo Liu\thanks{E-mail: jliu@math.duke.edu}}
\author[d]{Xiaoqian Xu\thanks{E-mail: xiaoqian.xu@dukekunshan.edu.cn}}

\affil[a]{School of  Mathematical Sciences, Shanghai Key Laboratory of PMMP, East China Normal University, Shanghai, 200241, P.R. China. }
\affil[b]{School of Mathematical Sciences, Institute of Natural Sciences, MOE-LSC, Shanghai Jiao Tong University, Shanghai, 200240, P.R.China.}
\affil[c]{Department of Mathematics, Department of Physics, Duke University, Durham, NC 27708, USA.}
\affil[d]{Zu Chongzhi Center for Mathematics and Computational Sciences, Duke Kunshan University, Kunshan, 215316, P.R.China}

\date{}

\begin{document}
\maketitle

\begin{abstract}
We study the existence of weak solutions to the $p$-Navier-Stokes equations with a symmetric $p$-Laplacian on bounded domains. We construct a particular Schauder basis in $W_0^{1,p}(\Omega)$ with divergence free constraint and prove existence of weak solutions using the Galerkin approximation via this basis.
Meanwhile, in the proof, we establish a chain rule for the $L^p$ norm of the weak solutions, which fixes a gap in our previous work. The equality of energy dissipation is also established for the weak solutions considered.  
\end{abstract}

\section{Introduction}

The system of Navier-Stokes equations is one of the most influential mathematical models in physical science and engineering fields \cite{tartar2006introduction}. The application of Navier-Stokes equations ranges from the design of a plane to weather forecasting. One of the Millennium Problems proposed by Clay Mathematics Institute is about the global existence of the smooth solutions to Navier-Stokes equations \cite{bblWebsite}, which remains one of the most important open questions in the field of partial differential equations \cite{lions1996mathematical}.

There are tons of models that are variants of the classical Navier-Stokes equations, typically for some Non-Newtonian fluids \cite{crochet2012numerical}. As an example, to study the shear thinning effect of the non-Newtonian flows, one could use the symmetric $p$-Laplacian term instead of Laplacian, and one may check \cite{vdbh2009,breit2017existence} for more discussion. 
In \cite{liliu17}, the authors proposed the $p$-Euler equations as the Euler-Lagrange equations from the Arnold's least action principle \cite{arnold1966geometrie,arnold2008topological}, for which the action is represented by the Benamou-Brenier characterization of the Wasserstein-$p$ distance between two shapes with the incompressibility constraint. By adding $p$-Laplacian diffusion to the equation, the so-called $p$-Navier-Stokes equations were proposed:
\begin{gather}\label{eq:pnsold}
\begin{split}
& \partial_t v_p+v\cdot\nabla v_p=-\nabla \pi+\nu \,\Delta_p v,\\ 
&v_p=|v|^{p-2}v,\quad \nabla\cdot v=0.
\end{split}
\end{gather}
Here, the $p$-Laplacian is given by $\Delta_p v=\nabla\cdot(|\nabla v|^{p-2}\nabla v)$, $|\nabla v|=\sqrt{\sum_{ij}(\partial_i v_j)^2}$.
Mathematically, the $p$-Navier-Stokes equations are analogues of the classical Navier-Stokes equations and exhibit many similar properties. In particular, 
when $p=2$, such a system becomes the classical Navier-Stokes equations. The generalization to general $p$, on the other hand, has some particular difficulty and fine structures. Due to the lack of Hilbert structure in $L^p$ space and the nonlinearity of all the terms in the differential equations, the analysis of such a system of differential equations is significantly more difficult than the classical problems.

In this paper, we are interested in the $p$-Navier-Stokes equations with the symmetric $p$-Laplacian arsing in the models for shear thinning effect \cite{vdbh2009} on a bounded domain $\Omega\subset \R^d$ with $C^{\infty}$ boundary $\partial \Omega$. In particular, we consider the initial-boundary value problem of the $p$-Navier-Stokes equations given by
\begin{gather}\label{pNSe}
	\left\{
	\begin{aligned}
		&\displaystyle\partial_t v_p+v\cdot\nabla v_p=-\nabla \pi+\nu \,\cL_p(v), & x\in\Omega,~t\in(0,T)\,,\\
		&v_p=|v|^{p-2}v,\quad \displaystyle\nabla\cdot v=0,
 &x\in\Omega,~t\in(0,T)\,,\\
		&\displaystyle v(x, 0)=v_0(x), &x\in\Omega\,,\\
		&\displaystyle v=0, & x\in\partial\Omega\,.
	\end{aligned}
	\right.
\end{gather} 
 Here, for the vector field $v$, the symmetric $p$-Laplacian $\cL_p(v)$ is 
\begin{gather}
\cL_p(v) = \div(|\mathcal{D}(v)|^{p-2}\mathcal{D}(v)),
\end{gather}
and
\begin{gather}
\cD(v)=\frac{1}{2}(\nabla v+\nabla v^{T}).
\end{gather}
We focus on the symmetric $p$-Laplacian because such a diffusion term appears in physical models \cite{vdbh2009}. We remark however that the analysis for the usual $p$-Laplacian diffusion $\Delta_p v$ would be similar (and in fact easier).  When $1<p<2$, $\Delta_p$  corresponds to fast diffusion. When $p>2$, it is the case corresponding to slow diffusion. One can check this in standard textbooks or references on $p$-Laplacian, for example \cite{vazquez2007porous}. One would expect the symmetric $p$-Laplacian term to exhibit similar diffusion effects.

The existence of weak solutions to the original $p$-Navier-Stokes equations \eqref{eq:pnsold} proposed in \cite{liliu17} has been explored in \cite{liliu17} and \cite{liu2021existence} by totally different methods. In \cite{liliu17}, a regularized system was proposed for the approximation and for the existence of weak solutions. Meanwhile, in \cite{liu2021existence}, the authors used the discrete time scheme to prove the existence of the weak solution. We remark that there are some minor gaps in the proof in \cite{liliu17}. For example, the well-posedness of the regularized system was taken for granted; second, the chain rule was not established rigorously as detailed in Section \ref{subsec:timereg}.

In this article, we focus on the equations \eqref{pNSe} with symmetric $p$-Laplacian on a bounded domain and establish the existence of weak solutions rigorously using a totally different method, the Galerkin approximation.
The reasons are as follows. First, the symmetric $p$-Laplacian is more frequently used for non-Newtonian fluids. Second, the well-posedness of the Galerkin system can be established rigorously. Moreover, we also aim to fill the gaps for the existence in the previous work.  The Galerkin method, or Galerkin approximation, is a very common method in numerical analysis as well as in applied analysis, especially for finding the local existence of the weak solutions to a particular differential equation. One can check more details in standard textbooks, for example \cite{evans2022partial}. In order to use the Galerkin approximation, one may need to choose a Hilbert space or a Banach space and find a Schauder basis. In this paper, by the natural structure of our differential equations \eqref{pNSe}, one has to use $L^p(\Omega)$ and $W_0^{1,p}(\Omega)$ spaces with divergence-free constraint as the reference spaces.  The study of the existence of Schauder basis on such spaces can be tracked back to \cite{fuvcik1972existence}. Later, in \cite{bellout95}, by connecting with the Haar system in one-dimensional case, the authors constructed a special Schauder basis with orthogonality properties on $W_0^{1,p}(\Omega)$. As we shall see later, due to the boundary condition, the Leray projection cannot be used directly to obtain the basis for the subspaces with divergence-free constraint. For the self-consistency of this paper, we construct a Schauder basis based on the eigenfunctions of a compact operator. The significance of this work can be summarized as follows. First, the existence of weak solutions is established rigorously using the Galerkin approximation fixing the previous gaps. In the proof, a chain rule for the $L^p$ integral of the weak solution is proved using the finite difference approximation, and this is also used to show the energy dissipation equality. Note that this technique can also be used to fill in the gap in \cite{liliu17} for the original model in $\R^d$. Second, a Schauder basis is constructed explicitly for $W_0^{1,p}(\Omega)$ with divergence free constraint. This can be used further for other models in $L^p$ type spaces. 

The structure of this paper is as follows. In Section \ref{2}, we introduce the notations and definitions in this paper. In Section \ref{3}, we construct a Schauder basis in the Sobolev spaces with divergence-free constraints. The basis consists of eigenfunctions of the projected high order elliptic operator in the space with divergence free constraint. In Section \ref{4}, we use the Galerkin approximation and run the compactness argument. In Section \ref{5}, we finish the proof of the main theorem. Here, the chain rule is established using finite time differences.

\section{Notations and definitions}\label{2}

Fix $\Omega\subset\mathbb{R}^d$ simply connected, bounded with $C^{\infty}$ boundary $\partial \Omega$. In the rest of this paper, we assume 
\begin{gather}
	p\geq d\geq 2.
\end{gather}

To make this paper self-consistent, we recall some notations in tensor analysis. Let $a, b\in \mathbb{R}^d$ be vectors and $A, B$ be matrices. We define $a\otimes b$ be a matrix, called the tensor product of $a$ and $b$:
\begin{gather}\label{eq:tensorproduct}
(a\otimes b)_{ij}=a_ib_j.
\end{gather}
We also define the dot product for vectors and matrices as
\begin{gather}\label{eq:tensordot}
\begin{split}
& (a\cdot A)_i=\sum_{j=1}^d a_j A_{ji},~~(A\cdot a)_i=\sum_{i=1}^d A_{ij}a_j, \\
& (A\cdot B)_{ij}=\sum_{k}A_{ik}B_{kj},\\
&A:B=\tr(A\cdot B^T)=\sum_{ij}A_{ij}B_{ij}.
\end{split}
\end{gather}

\subsection{The weak solutions}

To incorporate initial values in the definition of weak solution to \eqref{pNSe}, we introduce the following definition.
\begin{definition}\label{weaktime}
A function $f\in L^1[0,T]$ is said to have the weak time derivative $w\in (C_c^{\infty}[0,T))'$ with initial value $f_0$ if
\[
\int_0^T\phi w\,dt=-\int_0^T\phi' f\,dt-\phi(0)f_0,\text{~~}\forall \phi\in C_c^{\infty}([0,T)).
\]
A function  $f\in L_{\loc}^1(\Omega\times[0,T])$ is said to have weak time derivative $w\in (C_c^{\infty}(\Omega\times[0,T)))'$ and initial data $f_0(x)\in L_{\loc}^1(\Omega)$ if
\[
\int_0^T\int_{\Omega}\phi w\,dxdt=-\int_0^T\int_{\Omega}\partial_t\phi f \,dxdt-\int_{\Omega}\phi(x,0)f_0(x)\,dx,\text{~~}\forall \phi\in C_c^{\infty}(\Omega\times[0,T)).
\]
\end{definition}

We define the bounded trace operator $\mathcal{T}_r: W^{1,p}(\Omega;\mathbb{R}^d) \to L^p(\partial\Omega;\mathbb{R}^d)$, such that $\mathcal{T}_r(u)=u|_{\partial \Omega}$, for any $u\in C^{\infty}(\Omega;\mathbb{R}^d)$. Such operator is unique, one can check this fact from standard PDE textbooks, for instance \cite{evans2022partial}. We then denote $W^{1,p}_0(\Omega;\mathbb{R}^d)$ to be all the functions $u$ in $W^{1,p}(\Omega;\mathbb{R}^d)$ such that $\mathcal{T}_r(u)=0$. One could verify that the space $W^{1,p}_0(\Omega;\mathbb{R}^d)$ is the completion of $C^{\infty}_0(\Omega;\mathbb{R}^d)$ under the $W^{1,p}(\Omega;\mathbb{R}^d)$ norm. We further use $W^{-1,q}(\Omega;\mathbb{R}^d)$ to denote the the dual space of $W_0^{1,p}(\Omega; \mathbb{R}^d)$, where $q$ satisfies $\frac{1}{p}+\frac{1}{q}=1$.  

Following from Poincar\'{e} inequality, $\|v\|_{L^p}\leq c\|\nabla v\|_{L^p}$, for $v\in W_0^{1,p}(\Omega; \mathbb{R}^d)$. As a consequence, the norm in $W_0^{1,p}(\Omega; \mathbb{R}^d)$ is equivalent to $\|\nabla v\|_{L^p}+\|v\|_{L^p}$. Such norm is also equivalent to  $\|\mathcal{D}(v)\|_{L^p}+\|v\|_{L^p}$, for which the proof can be found in \cite{carlos1992}. Here we state the fact as the following lemma and prove it in Appendix \ref{app:equivnorm} for completeness.
\begin{lemma}
 There exist two positive constants $C_1, C_2$, such that for any $v\in W^{1,p}_0(\Omega;\R^d)$,
\begin{align}\label{eq:eqnorm}
C_1 (\|v\|_{L^p(\Omega)}+\|\nabla v\|_{L^p(\Omega)})\leq\|v\|_{L^p(\Omega)}+\|\mathcal{D}(v)\|_{L^p(\Omega)}\leq C_2 (\|v\|_{L^p(\Omega)}+\|\nabla v\|_{L^p(\Omega)})\,.
\end{align}
\end{lemma}
The action of symmetric $p$-Laplacian can be considered as an operator $\mathcal{L}_p$: $W_0^{1,p}(\Omega; \mathbb{R}^d)\to W^{-1,q}(\Omega; \mathbb{R}^d)$.  For $u,v \in W_0^{1,p}(\Omega; \mathbb{R}^d)$, define 
\begin{equation}\label{Deltap}
\langle \mathcal{L}_p(u), v\rangle :=-\int_\Omega|\mathcal{D}(u)|^{p-2}\mathcal{D}(u):\nabla v dx\,.
\end{equation}
In addition, by H\"older's inequality, one has 
\[
|\langle \mathcal{L}_p(u), v\rangle|=\left|-\int_\Omega |\mathcal{D}(u)|^{p-2} \mathcal{D}(u) : \nabla v dx\right|\leq\|\mathcal{D}(u)\|_p^{p-1}\|\nabla v\|_p,
\]
for $u, v \in W_0^{1,p}(\Omega; \mathbb{R}^d)$. Hence,
 \[
\int_0^T\|\mathcal{L}_p(u)\|_{W^{-1,q}(\Omega; \mathbb{R}^d)}^qdt\leq\int_0^T(\|\mathcal{D}(u)\|_p^{p-1})^{q}dt=\int_0^T\|\mathcal{D}(u)\|_p^pdt
\]
for $u \in L^p(0,T; W_0^{1,p}(\Omega; \mathbb{R}^d))$. This means that $\mathcal{L}_p$ maps bounded sets in $L^p(0,T; W_0^{1,p}(\Omega; \mathbb{R}^d))$ to bounded sets in $L^{q}(0,T; W^{-1,q}(\Omega; \mathbb{R}^d))$. 

To motivate the definition of weak solutions, let us perform some formal estimate ( {\it a priori} estimate).
Multiplying $v$ on both sides of the first equation in \eqref{pNSe} and integrating over space and time, one has 
\begin{gather}\label{eq:energyapriori}
\int_\Omega|v|^p(x,T)dx-\int_\Omega|v|^p(x,0)dx=-q\nu\int_0^T\int_\Omega|\mathcal{D}(v)|^pdxdt.
\end{gather}
Hence, if $v_0 \in L^p(\Omega;\mathbb{R}^d)$, one is expected to have
\begin{gather}
v\in L^{\infty}(0,T, L^p(\Omega;\mathbb{R}^d))\cap L^p(0,T; W_0^{1,p}(\Omega;\mathbb{R}^d)).
\end{gather}
We remark that the energy dissipation equality \eqref{eq:energyapriori} often reduces to inequality for weak solutions (we will show this equality holds later for our weak solutions). Nevertheless, the regularity of the $v$ with this {\it a priori} estimate is expected to hold.
Let $\hat v$ be the unit vector with the same direction as $v$. Based on the observation $\nabla v_p=|v|^{p-2}\nabla v\cdot (I+(p-2)\hat{v}\otimes\hat{v})$, we immediately obtain that in the case $p>2$
\begin{gather}
v_p\in L^{\infty}(0,T; L^q(\Omega; \mathbb{R}^d))\cap L^q(0,T; W_0^{1,q}(\Omega; \mathbb{R}^d))\,.
\end{gather}

By Definition \ref{weaktime}, with these a $priori$ estimates, it is natural for us to define the weak solutions to the initial-boundary $p$-Navier-Stokes problems as follows.
\begin{definition}\label{def:weak}
Given $v_0\in L^p(\Omega;\mathbb{R}^d)$ with $\int_{\Omega}\nabla \psi\cdot v_0\,dx=0$ for all $\psi\in C^{\infty}(\overline{\Omega}; \R)$, we say $v\in L^{\infty}(0,T, L^p(\Omega;\mathbb{R}^d)\cap L^p(0,T; W_0^{1,p}(\Omega;\mathbb{R}^d))$ is a weak solution of the $p$-Navier-Stokes problem (Equation (\ref{pNSe})) with initial value $v_0$, if 
\begin{equation}
\lim_{h\rightarrow 0^+}\int_0^{T-h}\|v(t+h)-v(t)\|_{L^p(\Omega;\mathbb{R}^d)}^pdt=0,
\end{equation}
and for any $\varphi\in C_c^{\infty}(\Omega\times [0,T);\mathbb{R}^d)$, $\nabla \cdot \varphi=0$, $\psi\in C_c^{\infty}(\overline{\Omega}\times[0,T);\mathbb{R})$, we have
\begin{gather}\label{ws}
\begin{split}
&\int_0^T\int_{\Omega}v_p\cdot \partial_t \varphi\, dxdt+\int_0^T\int_{\Omega}\nabla\varphi : (v\otimes v_p)dxdt-\nu\int_0^T\int_{\Omega}\nabla\varphi : \mathcal{D}(v)|\mathcal{D}(v)|^{p-2}dxdt\\
&+\int_{\Omega}|v_0|^{p-2}v_0\cdot\varphi(x,0)dx=0,\\
&\int_0^T\int_{\Omega}\nabla\psi\cdot v\, dxdt=0.
\end{split}
\end{gather}
If $v\in L_{\loc}^{\infty}(0,\infty;L^p(\Omega;\mathbb{R}^d))\cap L_{\loc}^p(0,\infty; W_0^{1,p}(\Omega;\mathbb{R}^d))$ and \eqref{ws} holds with $\infty$ instead of $T$
for all $\varphi \in C_c^{\infty}(\Omega\times[0,\infty);\mathbb{R}^d)$ with $\nabla \cdot \varphi=0$ and $\psi \in C_c^{\infty}(\overline{\Omega}\times[0,\infty);\mathbb{R})$, we say $v$ is a global solution.
\end{definition}
Above, following \eqref{eq:tensorproduct} and \eqref{eq:tensordot}, the double dots are interpreted as 
\begin{gather}
\begin{split}
&\nabla \varphi:(v\otimes v_p)=\sum_{ij}\partial_i\varphi_jv_i(v_p)_j,\\
&\nabla\varphi:\cD (v)=\sum_{ij}\partial_i\varphi_j(\cD(v))_{ij}.
\end{split}
\end{gather}

\subsection{The working subspaces}

For the convenience of the discussion, we aim to incorporate the divergence-free constraint into the working spaces. In particular, we need to seek a solution $v$ in the subspaces of $L^p(\Omega; \R^d)$ and $W_0^{1,p}(\Omega)$ with certain divergence-free constraints. Define
\begin{gather}
\cU:=\{\phi\in C_c^{\infty}(\Omega;\mathbb{R}^d): \div\phi=0\}.
\end{gather}
For the space $L^p(\Omega; \R^d)$, we recall the \textit{Helmohotlz-Weyl} decomposition \cite{galdi2011}.   Denote $U_p(\Omega)$ the $L^p$-completion of the space $\cU$, which is given by
 \begin{gather}\label{Up}
 U_p(\Omega)=\left\{w\in L^p(\Omega;\mathbb{R}^d): \int_{\Omega} w\cdot\nabla\varphi\,dx=0, \forall \varphi \in C^1(\overline{\Omega})\right\}.
 \end{gather}
 This is the weak form of $\{w\in L^p(\Omega;\mathbb{R}^d): \nabla\cdot w=0\text{ in } \Omega,\, w\cdot n=0
\text{ on } \partial \Omega\}$, where $n$ represents normal vector field on the boundary.  Let $G_p(\Omega)=\{w\in L^p(\Omega,\mathbb{R}^d): \exists \varphi \in W_{\mathrm{loc}}^{1,p}(\Omega), w=\nabla\varphi\}$. Theorem \uppercase\expandafter{\romannumeral3}.1.2 and relevant results in \cite{galdi2011} can be summarized as the following lemma:
\begin{lemma}\label{Helm}
	Let $\Omega \subset \mathbb{R}^d$, $d\geq 2$ be either a domain of class $C^2$ or the whole space or a half space, then the Helmholtz-Weyl decomposition holds,
    \begin{gather}
	L^p(\Omega; \R^d)=U_p(\Omega)\oplus G_p(\Omega)\,,
	\end{gather}
  where $\oplus$ denotes direct sum. This defines the Leray projection operator $\mathcal{P}: L^p(\Omega; \R^d)\rightarrow U_p(\Omega)$. There is a constant $C(p,\Omega)$ such that for any $w\in L^p(\Omega;\R^d)$, 
	\begin{gather}
	\|\mathcal{P}w\|_p\leq C(p,\Omega)\|w\|_p.
	\end{gather}
\end{lemma}
This says that  any  $w\in L^p(\Omega;\mathbb{R}^d)$ can be uniquely decomposed as
 \begin{align*}
 w=w_1+w_2\,,
\end{align*}
 with $w_1\in U_p(\Omega)$, and $w_2\in G_p(\Omega)$ and thus $w_1=\cP w$.   The decomposition here is the so-called Helmholtz-Weyl decomposition.
 \begin{remark}
The boundary condition matters. For example, $\phi=(-y, x)$ is divergence free in $\Omega=\{(x, y): 2x^2+y^2<1\}$ but $\phi\cdot n\neq 0$ on $\partial\Omega$. Then, $\cP\phi\neq \phi$ since $\int_{\Omega} \nabla \varphi\cdot \phi\,dx\neq 0$ for some $\varphi$.
\end{remark}

 For $W_0^{1,p}$, since the weak derivatives are well-defined, we can introduce directly
\begin{gather}
W :=\left\{v\in W_0^{1,p}(\Omega;\mathbb{R}^d): \nabla \cdot v=0 \right\}.
\end{gather}
Here $\nabla\cdot$ means divergence in the weak sense. Moreover, we will also use
\begin{gather}\label{spV}
V:=L^p(0, T; W),
\end{gather}
equipped with the $L^p(0, T; W^{1,p})$ norm.

\begin{remark}
As in Lemma \ref{Helm}, for any $\phi \in W_0^{1,p}$, the Helmholtz-Weyl decomposition of $\phi$ is given by
\[
\phi=\cP\phi+\nabla\varphi,
\]
	where $\cP\phi\in U_p(\Omega)$, $\nabla \varphi \in G_p(\Omega)$, and $\varphi$ is unique up to a constant. Then, $\varphi$ can be determined by the following Poisson equation
	\begin{gather*} 
 \left\{
\begin{split}
&\Delta \varphi=\nabla\cdot \phi \text{~~~in }\Omega,\\
&\frac{\partial \varphi}{\partial n}=0 \text{~~~on }\partial \Omega.
\end{split}
 \right.
	\end{gather*}
Clearly, one has $\cP \phi \in W^{1,p}$ by the elliptic regularity. Unfortunately, the boundary value of $\cP$ is not necessarily zero (only the normal component is zero).
Hence, the projection of $v\in W_0^{1,p}$ onto $W$ {\it cannot} be simply obtained using the Leray projection.
\end{remark}

With the spaces in hand, clearly, we will then seek solutions in 
$L^{\infty}(0, T; U_p(\Omega))\cap V$.
For the aim of this purpose, we need a Schauder basis for $U_p(\Omega)$ and $W$.

\section{A Schauder basis}\label{3}

The existence of the Schauder basis of $W_0^{1,p}$ is well-known (see, for example \cite{bellout95}). However, as commented in the last section, one can not simply use the Leray projection $\cP$ to obtain a basis for $W$.
In this section, we will show that the eigenfunctions in certain spaces of $\cP\Delta^m$, if $m$ large enough, will form a Schauder basis for both $U_p(\Omega)$ and $W$.
 
We first of all consider the following elliptic problem
\begin{gather}\label{ellp}
\begin{split}
&(-1)^m\cP \Delta^m u=f,\\
& \Delta^s u|_{\partial \Omega}=0,\quad \frac{\partial}{\partial n}\Delta^{\ell}u|_{\partial\Omega}=0, \quad \forall s\in S, ~\forall\ell\in L.
\end{split}
\end{gather}
Here $f\in U_2(\Omega)$, the space which is $L^2$-completion of divergence-free smooth functions as defined in \eqref{Up}. If $m=2k$, then $S=\{s\in\mathbb{Z}:0\le s\le k-1\}$ and $L=\{\ell\in\mathbb{Z}:0\le \ell\le k-1\}$; if $m=2k+1$, then $S=\{s\in\mathbb{Z}:0\le s\le k\}$, $L=\{\ell\in\mathbb{Z}:0\le \ell \le k-1\}$. 
We use $U_2(\Omega)$ here to make use of its Hilbert structure. The domain of the operator $\mathcal{A}=(-1)^m \cP \Delta^m: U_2(\Omega)\to U_2(\Omega)$ is given by
\[
\cD(\cA)=H^{2m}\cap  \tilde{H}^m,
\]
where
\begin{gather*}
\tilde{H}^m=\left\{u\in H^m: \mathrm{div}(u)=0,~\Delta^s u|_{\partial\Omega}=0, ~\frac{\partial}{\partial n}\Delta^{\ell}u|_{\partial\Omega}=0, ~\forall s\in S,~\forall \ell\in L \right\}.
\end{gather*}
We remark that $\cU$ is not dense in Hilbert space $\tilde{H}^m$ because the completion of $\cU$ is $\{u\in H_0^m: \mathrm{div}(u)=0\}$.

Now, consider the weak solution to the problem \eqref{ellp}. The associated bilinear form $B: \tilde{H}^m\times \tilde{H}^m \to \mathbb{R}$ is given by
\begin{gather}\label{defB}
B[u, v]=
\begin{cases}
\int \Delta^k u \Delta^kv\,dx, & m=2k,\\
\int \nabla \Delta^k u \cdot \nabla \Delta^k v\,dx, & m=2k+1.
\end{cases}
\end{gather}
A weak solution $u\in \tilde{H}^m$ is the one for which
\[
B[u, v]=\int_{\Omega}fv\,dx, \quad \forall v\in \tilde{H}^m.
\]
We remark that this definition of weak solution is consistent with problem \eqref{ellp}.
In fact, a weak solution is called a strong solution if the left hand side is a locally integrable function and \eqref{ellp} holds for a.e. $x$.   If a weak solution $u\in H^{2m}$, then integration by parts gives
\[
\int_{\Omega}(-1)^m\Delta^m u v\,dx= \int_{\Omega}(-1)^m\cP\Delta^m u v\, dx=\int_{\Omega} fv\,dx,
\]
for all $v\in \tilde{H}^m$. Hence, $u$ is a strong solution of \eqref{ellp}.

By the Lax-Milgram theorem (see \cite[Chapter 6]{evans2022partial}), the existence and uniqueness of the weak solution hold. Hence the solution map 
\[
\cS:=((-1)^m\cP\Delta^m)^{-1}: U_2(\Omega)\to \tilde{H}^m \subset U_2(\Omega)
\]
is well-defined.

\begin{remark}
The bilinear form for the equation $(-1)^m\Delta^m u=f$ with the same boundary conditions has the same expression, but the domain is $\hat{H}^m\times \hat{H}^m$ with $\hat{H}^m=\{u\in H^m, ~\Delta^s u|_{\partial\Omega}=0, ~\frac{\partial}{\partial n}\Delta^{\ell}u|_{\partial\Omega}=0, ~\forall s\in S, ~\forall \ell\in L\}$.
Note that though $(-1)^m\cP\Delta^m$ and $(-1)^m\Delta^m$ agree on $\cU$, they do not agree as maps $\tilde{H}^m\to (\hat{H}^m)'$ (they are identical as $\tilde{H}^m \to (\tilde{H}^m)'$ though). Here prime means the dual space.  This suggests that $((-1)^m\cP \Delta^m)^{-1}f$ is different from $((-1)^m  \Delta^m)^{-1}f$ as elements in $\hat{H}^m$ when $f\in U_2(\Omega)$. In particular, $((-1)^m  \Delta^m)^{-1}f$ may not be divergence free even if $f$ is.
\end{remark}

\begin{proposition}\label{pro:basis}
The eigenfunctions of $\cS$ in $U_2(\Omega)$ form a Schauder basis for both $U_p\Omega)$ and $W$
if $m$ is sufficiently large.
\end{proposition}

\begin{proof}
The operator $\cS$ is self-adjoint and compact as a map from $U_2(\Omega)$ to $U_2(\Omega)$. Then, it has a complete set of eigenfunctions in $U_2(\Omega)$.  
Denote the set of eigenfunctions as $\{\phi_k\}_{k=1}^{\infty}$ and the corresponding eigenvalues as $\{\lambda_k\}_{k=1}^{\infty}$. They are orthogonal in $L^2$, and thus they form a Schauder basis for $U_2(\Omega)$.

First, we show that the eigenfunctions form a Schauder basis for $\tilde{H}^m$. This is done by the same argument as in the proof of \cite[Section 6.5, Theorem 2]{evans2022partial}. In fact, by the elliptic regularity, one can show that $\phi_k \in H^{2m}$. Hence,  $\phi_k$ are also the eigenfunctions of $(-1)^m\cP\Delta^m$.
For any $u\in \tilde{H}^m$, by definition of $B$ in \eqref{defB} and integration by parts, one has
\[
B[u, \phi_k]=\int_{\Omega}u (-1)^m\Delta^m\phi_k
=\int_{\Omega}u (-1)^m\cP\Delta^m\phi_k=\lambda_k^{-1}\int_{\Omega}u\phi_k\,dx.
\]
Since $\{\phi_k\}$ form a basis for $U_2(\Omega)$, then $B[u,\phi_k]=0$ for all $k\ge 1$ implies that $u=0$ in $U_2(\Omega)$ and thus in $\tilde{H}^m$. Hence, $\{\phi_k\}$ is also complete in $\tilde{H}^m$. Moreover, it is orthogonal as well in $\tilde{H}^m$; hence, it is a Schauder basis for $\tilde{H}^m$.
The convergence in $\tilde{H}^m$ clearly implies the convergence in $U_2(\Omega)$. Hence, the expansion coefficient is the same in the two spaces.

Second, choose $m$ sufficiently large such that $\tilde{H}^m \subset W \subset U_2(\Omega)$, and the embeddings are continuous due to the Sobolev inequalities. 

It is clear that $\cU \subset \cS (U_2(\Omega)) \subset \tilde{H}^m \subset W$ where all the embeddings are continuous. Hence, $\cS(U_2(\Omega))$ is dense both in $U_p(\Omega)$ and $W$. Let $\cX$ be $U_p(\Omega)$ or $W$, and $\|\cdot\|$ is the corresponding norm. 

For every $u\in \cS(U_2(\Omega))\subset{W}$, one has the expansion in $\tilde{H}^m$ and thus
\[
u=\sum_{k=1}^{\infty} c_k \phi_k,\quad \text{in}\quad \cX.
\]
Consider the projection operator on $\cS(U_2(\Omega))$
\[
P_{m,m'}u:=\sum_{k=m}^{m'} c_k \phi_k.
\]
One has
\[
\|P_{m,m'}u\| \le C(m,m')\sqrt{\sum_{k=m}^{m'}c_k^2}
\le C(m,m')\|u\|_{L^2}\le \tilde{C}(m,m',\Omega,p)\|u\|.
\]
The first inequality is by the equivlance of norms for finite dimensional space, while others are trivial. Hence, $P_{m,m'}$ can be extended to the whole $\cX$. Since
\[
\|P_{m,m'}u\| 
\le \left\|\sum_{k=1}^{m'}c_k\phi_k\right\| +\left\|\sum_{k=1}^{m}c_k\phi_k\right\| 
\,,\]
and $\sum_{k=1}^mc_k\phi_k$ converges to $u$ in $\cX$, then the trajectory $O_u:=\{ P_{m,m'}u :  1\le m\le m'<\infty \}$ is bounded. By the Uniform Boundedness Principle,
\[
\sup_{1\le m\le m'<\infty}\|P_{m,m'}\| <\infty.
\]
Now, for any $u_*\in \cX$, we take a sequence $u_n \in \cS(U_2(\Omega))$ such that  $u_n\to u_*$ in $\cX$, which can be expressed by 
\[
u_n=\sum_{k=1}^{\infty}c_{nk}\phi_k \text{~in $\cX$}\,.
\]
 Then, for any $\epsilon>0$, there exists $n_0>0$ such that whenever $n_2>n_1\ge n_0$,
\[
\sup_{m,m'}\|P_{m,m'}(u_{n_1}-u_{n_2})\|\le C\|u_{n_1}-u_{n_2}\|< \epsilon.
\]
This implies that $c_{nk}\to \bar{c}_k$. Moreover, $\|P_{m,m'}u_{n_2}\|\le \|P_{m,m'}u_{n_1}\|+\epsilon$. Fixing $n_1$, taking $m$ large enough and taking $n_2\to\infty$, one then has $\|\sum_{k=m}^{m'}\bar{c}_k\phi_k\|<2\epsilon$. Then $\sum_{k=1}^m \bar{c}_k \phi_k$ is a Cauchy sequence in $\cX$. Hence, 
\[
\bar{u}=\sum_{k=1}^{\infty}\bar{c}_k\phi_k \in \cX.
\]
It is easy to identify $\bar{u}$ with $u_*$. This means that $\{\phi_k\}$ is a Schauder basis for $\cX$ as well and the expansion coefficient should be the same as in $U_2(\Omega)$ since the embedding from $\cX$ to $U_2(\Omega)$ is continuous.
\end{proof}

\section{The Galerkin approximation and precompactness}\label{4}

In this section, we apply the Galerkin approximations to \eqref{pNSe} and perform the energy estimates. Then, we obtain the precompactness of the solutions to the Galerkin systems.

\subsection{Galerkin approximation}
To introduce the Galerkin's approximation, for any $v_0\in U_p(\Omega)$, we write it as 
\[
v_0=\sum_{n\ge 0} c_{0,n}\phi_n\quad \text{in}\quad U_p(\Omega).
\]
Here $\{\phi_n\}$ is the Schauder basis we constructed in the last section. Since the case for $v_0=0$ is trivial, we consider the case $v_0\neq 0$. Hence there is a minimum $n_*$ such that $c_{0,,n_*}\neq 0$.
For all $N\ge n_*$, let 
\begin{equation}
W_N=\mathrm{span}\{\phi_1, \ldots, \phi_N \}.
\end{equation}
We hope to obtain a function $v^N:$ $[0,T]\to W_N \subset W$ of the form
 \[
 v^N(t)=\sum_{n=1}^N c_n^N(t)\phi_n,
 \]
where the coefficients $c_n^N(t)\in \R$ $(0\le t\le T, n=1,\cdots,N)$ satisfy
\begin{enumerate}[(i)]
\item The initial conditions hold for $0\le n\le N$:
\begin{gather}\label{galerkin1}
c_n^N(0)=c_{0,n}.
\end{gather}

\item For any $0\le t\le T$, $\varphi \in W_N$, 
\begin{gather}\label{galerkin2}
\frac{d}{dt}\int_{\Omega}\varphi	\cdot v_p^Ndx+\int_{\Omega}\varphi \cdot(v^N\cdot \nabla v_p^N)dx+\nu\int_{\Omega}\nabla\varphi:\mathcal{D}(v^N)|\mathcal{D}(v^N)|^{p-2}dx=0.
\end{gather}
Here similar to \eqref{pNSe}, 
\begin{align}
v_p^N=|v^N|^{p-2}v^N\,.
\end{align}
\end{enumerate}

Clearly, the equation \eqref{galerkin2} holds if for $i=1,...,N$, 
\begin{gather}\label{galerkin3}
\frac{d}{dt}\int_{\Omega}\phi_i	\cdot v_p^Ndx+\int_{\Omega}\phi_i\cdot(v^N\cdot \nabla v_p^N)dx+\nu\int_{\Omega}\nabla\phi_i:\mathcal{D}(v^N)|\mathcal{D}(v^N)|^{p-2}dx=0.
\end{gather}

The term
$\frac{d}{dt}\int_{\Omega}\phi_i	\cdot v_p^Ndx$  is equal to 
\begin{gather}
\sum_{j=1}^{N} \frac{d}{dt}c_j^N(t) \int_{\Omega}|v^N|^{p-2}\phi_i^T(I+(p-2)\hat{v}^N\otimes\hat{v}^N)\phi_j dx,
\end{gather}
where $\hat{v}^N$ is the unit vector with the same direction as $v^N$.

The term $ \int_{\Omega}\phi_i\cdot(v^N\cdot \nabla v_p^N)dx$ is equal to
\begin{equation}
\sum_{j,k=1}^N\int_{\Omega}|v^N|^{p-2}\phi_{i,k}v_j^N\partial_jv_k^Ndx+(p-2)\sum_{j,k,l=1}^N\int_{\Omega}|v^N|^{p-4}\phi_{i,k}v_j^N\partial_jv_l^Nv_l^Nv_k^Ndx.
\end{equation}
Denote 
\begin{gather}
A_{ij}^N(t) :=\int_{\Omega}|v^N|^{p-2}\phi_i^T(I+(p-2)\hat{v}^N\otimes\hat{v}^N)\phi_j dx,
\end{gather}
and
\begin{gather}
X^N(t):=\left(\begin{array}{c}c_1^N(t)\\ \vdots \\c_N^N(t)
\end{array}\right).
\end{gather}
Then the system~\eqref{galerkin2} is reduced to following equation,
\begin{gather}\label{eq:ode}
\left\{\begin{split}&A^N\dot{X}^N(t)=F(X^N(t)) ,\\
	&X^N(0)=X_0^N.\end{split}\right.
\end{gather}
Here $F(X^N(t))$ is a vector valued function in $\mathbb{R}^N$ with
\begin{multline}\label{e:Fform} 
(F(X^N))_i=-\nu\int_{\Omega}\nabla\phi_i:\mathcal{D}(v^N)|\mathcal{D}(v^N)|^{p-2}dx- \sum_{j,k=1}^N\int_{\Omega}|v^N|^{p-2}\phi_{i,k}v_j^N\partial_jv_k^Ndx\\
-(p-2)\sum_{j,k,l=1}^N\int_{\Omega}|v^N|^{p-4}\phi_{i,k}v_j^N\partial_jv_l^Nv_l^Nv_k^Ndx.
\end{multline}

\begin{lemma}\label{lmm:FAlocal}
As long as $X^N\neq 0$, $F(X^N)$ is locally Lipschitz in $X^N$ and $A^N$ is positive definite .	
\end{lemma}
\begin{proof}

From \eqref{e:Fform}, it is straightforward that $F(X^N)$ is $C^1$ in $X^N$ as long as $X^N \neq 0$.

Pick any vector $a\in \mathbb{R}^N$, $a\neq 0$, one have the following expression of $a^TA^Na$
\[
\int_{\Omega}|v^N|^{p-2}\sum_{i=1}^Na_i\phi_i^T(I+(p-2)\hat{v}^N\otimes\hat{v}^N)\sum_{j=1}^Na_j\phi_j \,dx.
\]
Denote $\alpha=\sum_{i=1}^Na_i\phi_i$, we thus have
\[
\int_{\Omega}|v^N|^{p-2}\alpha^T(I+(p-2)\hat{v}^N\otimes\hat{v}^N)\alpha\, dx.
\]
We notice that the engienvectors of $\hat{v}^N\otimes\hat{v}^N$ are vectors parallel to $v^N$ and vectors perpendicular to $v^N$, so the eigenvalues of $\hat{v}^N\otimes\hat{v}^N$ are 1 and 0. Consequently the eigenvalues of $(I+(p-2)\hat{v}^N\otimes\hat{v}^N)$ are $p-1$ and 1. Hence we have 
\[\int_{\Omega}|v^N|^{p-2}\alpha^T(I+(p-2)\hat{v}^N\otimes\hat{v}^N)\alpha \,dx\geq \min\{p-1,1\}\int_{\Omega}|v^N|^{p-2}|\alpha|^2\,dx>0,
\]
as long as $v^N\neq 0$, i.e. $X^N\neq 0$.
\end{proof}

\begin{proposition}\label{pro:localexis}
Given $|X|\neq 0$, there exists $\delta>0$ and a unique $X^N(t)\in C^1([0,\delta))$ such that $|X^N(t)|>0$ satisfying ~\eqref{eq:ode}, and $X^N(0)=X$.  
\end{proposition}
\begin{proof}
First by Lemma \ref{lmm:FAlocal}, we can rewrite ODE  ($\ref{eq:ode}$) as $\dot{X}^N(t)=(A^N)^{-1}F(X^N)$. By \textit{Cramer's rule}, \[
(A^N)^{-1}=\frac{1}{\det(A^N)}M^T,
\]
where $M$ is the matrix of \textit{cofactors} of $A^N$. Since $\det(A^N)$ and $M$ are both $C^1$ in $X^N$ as long as $X^N\neq 0$, we have $(A^N)^{-1}F(X^N)$ is $C^1$ as long as $X^N\neq 0$. This ensures $(A^N)^{-1}F(X^N)$ is locally Lipchitz  as long as $X^N\neq 0$. Hence following from classical ODE theory, we conclude that there exists $\delta>0$ such that ODE system (\ref{eq:ode}) has a unique solution on $[0,\delta)$.
\end{proof}

Note that by the argument in the proof of Proposition \ref{pro:localexis}, as long as $0<|X_N(t)|<\infty$, the solution can be extended. The largest existence time $t_*$ before $|X_N|$ touching $0$ is thus defined by
\begin{align}
t_*:=\sup\{t\ge 0: \eqref{eq:ode} \text{~has a unique solution~} X^N\in C^1[0, t], |X^N(s)|\neq 0\,,~\forall s\in [0, t)\}.
\end{align}
Clearly, at least one of the followings must happen if $t_*<\infty$:
\begin{itemize}
\item $\limsup_{t\to t_*}|X_N(t)|=+\infty$;
\item $\liminf_{t\to t_*}|X_N(t)|=0$.
\end{itemize}

Next we prove that $X^N(t)$ is never 0 and does not blow up . Once this has been proved, the solution to ODE ($\ref{eq:ode}$) is defined globally.

\begin{proposition}\label{vbound}
Suppose $v_0 \in U_p(\Omega)$. For $T<t_*$, one has
\begin{gather}\label{eq:energy}
\frac{d}{dt}\int_\Omega|v^N|^pdx=-q\nu\int_\Omega|\mathcal{D}(v^N)|^pdx
\end{gather}
and thus there exists a constant $C(p,\nu, v_0)$ independent of $N$ and $T$ such that
\begin{gather}
\begin{split}
&\|v^N\|_{L^\infty(0,T;L^p(\Omega))}\le \|v_0^N\|_{L^p(\Omega)},\\
&\|v^N\|_{L^p(0,T;W_0^{1,p}(\Omega))}\le C(p,\nu,v_0).
\end{split}
\end{gather}
Moreover, there are positive constants $C,C_N$ such that $\int_\Omega|v^N|^pdx\geq C e^{-C_N t}$ for any $t\le T$. Consequently, the solution $v^N$ exists globally (i.e., $t_{*} =\infty$).
\end{proposition}
\begin{proof}
First take $\varphi = v^N$ in \eqref{galerkin2} (equivalently, multiply $X^N(t)^T$ on both sides of $\eqref{eq:ode}$). As $v^N$ is divergence free and disappears on the boundary, one has $\langle v^N, v^N\cdot\nabla v_p^N\rangle=0$ . 
Moreover,
\[
\nabla v^N:\cD(v^N)=\cD(v^N):\cD(v^N).
\]
Hence, we have
\[
\frac{d}{dt}\int_\Omega|v^N|^pdx=-q\nu\int_\Omega|\mathcal{D}(v^N)|^pdx,
\]
where $q$ satisfies $\frac{1}{p}+\frac{1}{q}=1$.
As a result, we have $\|v^N(t)\|_{L^p(\Omega)}\le \|v_0\|_{L^p(\Omega)}$ for any $0\le t\le T$, or $\|v^N\|_{L^\infty(0,T;L^p(\Omega))}\le \|v_0\|_{L^p(\Omega)}$.
Integrating equation (\ref{eq:energy}) over time interval $[0,T]$, one have
\begin{gather*}
\|v^N(T)\|_{L^p(\Omega)}^p-\|v^N(0)\|_{L^p(\Omega)}^p=-q\nu\|\mathcal{D}(v^N)\|_{L^p(0,T;L^p(\Omega))}.
\end{gather*}
This implies
\begin{gather*}
\begin{split}
\|v^N\|_{L^p(0,T;W_0^{1,p}(\Omega))} &\leq C \|\mathcal{D}(v^N)\|_{L^p(0,T;L^p(\Omega))} \\
& =\frac{C}{q\nu}\big(\|v^N(0)\|_{L^p(\Omega)}^p-\|v^N(T)\|_{L^p(\Omega)}^p\big)\le \frac{C}{q\nu}\|v_0\|_{L^p(\Omega)}^p.
\end{split}
\end{gather*}

Next, we show that $t_*=\infty$. In fact, define
\[
\|X^N\|_1:=\left(\int_\Omega|v^N|^pdx\right)^{1/p}
\]
and
\[
\|X^N\|_2:=\left(\int_\Omega|\mathcal{D}(v^N)|^pdx\right)^{1/p}.
\]
It is easy to see by Minkowski inequality that both $\|\cdot\|_1$ and $\|\cdot\|_2$
are norms for $X^N$. Since $X^N$ is in a finite dimensional Euclidean space, one thus can find a constant $c_N^1>0, c_N^2>0$ such that
\[
c_N^1\|X^N\|_2\le \|X^N\|_1 \le c_N^2\|X^N\|_2,
\quad\forall X^N\in \R^{N}.
\]
Hence,
\[
\frac{d}{dt}\int_\Omega|v^N|^pdx\ge -\frac{q\nu}{(c_N^1)^p} \int_\Omega|v^N|^pdx.
\]
Hence, $v^N$ is never zero.
Moreover, since $\|v^N\|_{L^\infty(0,T;L^p(\Omega))}\le \|v_0^N\|_{L^p(\Omega)}$, $X^N$ never blows up so that one can in fact take $t_*=\infty$.
\end{proof}

By the fact that $v_p^N=|v^N|^{p-2}v^N$, it is easy to obtain the following Corollary.
\begin{corollary}\label{vpbound}
It holds that 
\[
v_p^N\in L^{\infty}(0,T;L^q(\Omega;\mathbb{R}^d))\cap L^q(0,T;W_0^{1,q}(\Omega;\mathbb{R}^d)).
\]
Moreover, 
\[
\sup_N \|v_p^N\|_{L^{\infty}(0, T; L^q)}
+\|v_p^N\|_{L^q(0, T; W^{1,q})}<\infty.
\]
\end{corollary}

\subsection{Compactness}

In this section, we prove the precompactness of the sequences generated by the Galerkin approximation~\eqref{galerkin1} and~\eqref{galerkin3}.

Later, we will need the time regularity of the sequences. To this end, we introduce a time-shift operator:
\begin{equation}
\tau_h v^N(x,t)=v^N(x,t+h).
\end{equation}

First we state a lemma which is useful in proving the convergence of time-shift operator. For a detailed proof of the lemma, one can read \cite{dl1998}  or \cite[Lemma 2]{liliu17}.
\begin{lemma}\label{gady}
Let $p>1$, then there exists $C(p)>0$ such that for any $\eta_1, \eta_2 \in \mathbb{R}^d$, it holds that
\[
(|\eta_1|^{p-2}\eta_1-|\eta_2|^{p-2}\eta_2)\cdot(\eta_1-\eta_2)\geq C(p)(|\eta_1|+|\eta_2|)^{p-2}
|\eta_1-\eta_2|^2.\]
\end{lemma}

In the following lemma, we would study the asymptotic behavior of sequence $\tau_hv^N(x,t)$ as $h$ goes to 0.
\begin{lemma} \label{shiftest}
Let $p\geq d\geq 2$ and $\Omega$ be a bounded domain, then it holds that
	$\|\tau_hv^N-v^N\|_{L^p(0,T-h;L^p(\Omega;\mathbb{R}^d))}\to 0$ uniformly in $N$ as 	$h\to 0+$.
\end{lemma}
\begin{proof}
For any fixed $t\leq T-h$ and any $\varphi\in W_N$, one has
\begin{equation}\label{eq:shift}
\left\l\tau_hv_p^N(t)-v_p^N(t),\varphi\right\r+\left\l\int_t^{t+h}v^N\cdot\nabla v_p^Nds,\varphi\right\r= \nu\left\l\int_t^{t+h}\mathcal{L}_p(v^N)ds,\varphi\right\r\, .
\end{equation}
Taking $\varphi=\tau_hv^N(t)-v^N(t)$, we now estimate each term in \eqref{eq:shift} in detail. 
 
 First, by Lemma $\ref{gady}$, it holds that
\begin{gather}
\begin{split}
\int_{\Omega}(\tau_hv_p^N(t)-v_p^N(t))\cdot(\tau_hv^N(t)-v^N(t))dx&\geq C(p)\int_{\Omega}(|\tau_hv^N|+|v^N|)^{p-2}|\tau_hv^N-v^N|^2dx\\
&\geq C(p)\|(\tau_hv^N-v^N)(t)\|_{p}^p.
\end{split}
\end{gather}

For the term $\int_{\Omega}\int_t^{t+h}\tau_hv^N(t)\cdot\mathcal{L}_p(v^N)dsdx$, Young's inequality yields:
\begin{gather}
\begin{split}
\int_{\Omega}\int_t^{t+h}\tau_hv^N(t)\cdot\mathcal{L}_p(v^N)dsdx&=-\int_t^{t+h}\int_{\Omega}(\mathcal{D}(\tau_hv^N):\mathcal{D}(v^N))|\mathcal{D}(v^N)|^{p-2}dxds\\
&\leq\int_t^{t+h}(\frac{1}{p}\|\mathcal{D}(\tau_hv^N)(t)\|_p^p+\frac{1}{q}\|\mathcal{D}(v^N)(s)\|_p^p)ds\\
&=\frac{h}{p}\|\mathcal{D}(\tau_hv^N)(t)\|_p^p+\frac{1}{q}\int_t^{t+h}\|\mathcal{D}(v^N)(s)\|_p^pds,
\end{split}
\end{gather}
where $\frac{1}{p}+\frac{1}{q}=1$.

Similarly for the integral term $\int_{\Omega}\int_t^{t+h}v^N(t)\cdot\mathcal{L}_p(v^N)dsdx$, one has
\begin{gather}
\int_{\Omega}\int_t^{t+h}v^N(t)\cdot\mathcal{L}_p(v^N)dsdx\leq \frac{h}{p}\|\mathcal{D}(v^N)(t)\|_p^p+\frac{1}{q}\int_t^{t+h}\|\mathcal{D}(v^N)(s)\|_p^pds.
\end{gather}

To estimate the term $\int_{\Omega}\int_t^{t+h}v^N\cdot\nabla v_p^N\cdot\tau_hv^N(t) dsdx$, we need the Gagliardo-Nirenberg inequality, which tells us that on a bounded domain $\Omega \times [0, T]$, for any function $f\in L^\infty(0,T;L^p(\Omega;\R^d))\cap L^p(0,T;W^{1,p}(\Omega;\R^d))$,
\begin{gather*}
\|f\|_{2p}^{2p}\leq C\|\nabla f\|_p^d\|f\|_p^{2p-d}
\leq C(\|f\|_{L^{\infty}(0,T; L^p(\Omega; \mathbb{R}^d))})\|\nabla f\|_p^d.
\end{gather*}
Therefore, 
\begin{gather}
\begin{split}
& \int_{\Omega}\int_t^{t+h}v^N\cdot\nabla v_p^N\cdot \tau_hv^N(t) dsdx \\
&\leq\int_t^{t+h}\frac{1}{2p}(\|\tau_hv^N(t)\|_{2p}^{2p}+\|v^N(s)\|_{2p}^{2p})+\frac{1}{q}\|v_p^N(s)\|_{W^{1,p}}^qds\\
&\leq C\left(h\|\nabla\tau_hv^N(t)\|_p^d+\int_t^{t+h}\|\nabla v^N(s)\|_p^d+\|v_p^N(s)\|_{W^{1,q}}^qds \right).
\end{split}
\end{gather}

Similarly, it holds that
\begin{gather}
\int_{\Omega}\int_t^{t+h}v^N\cdot\nabla v_p^N\cdot v^N(t)dsdx \leq C\left(h\|\nabla v^N(t)\|_p^d+\int_t^{t+h}\|\nabla v^N(s)\|_p^d+\|v_p^N(s)\|_{W^{1,q}}^qds\right).
\end{gather}

Overall, we have the final estimate:
\begin{gather}
\begin{split}
\|(\tau_hv^N-v^N)(t)\|_{p}^p
&\leq Ch(\|\nabla\tau_hv^N(t)\|_p^p+\|\nabla v^N(t)\|_p^p+|\nabla\tau_hv^N(t)\|_p^d+\|\nabla v^N(t)\|_p^d)\\
&+C\int_t^{t+h}\|\nabla v^N(s)\|_p^p+\|\nabla v^N(s)\|_p^d+\|v_p^N(s)\|_{W^{1,q}}^qds.
\end{split}
\end{gather}

Integrating both sides over time $t$ from 0 to $T-h$, one has
 \begin{gather}
 \begin{split}
&\|\tau_hv^N-v^N\|_{L^p(0,T-h;L^p(\Omega;\mathbb{R}^d))}^p\\
 &\leq C_1h\int_0^{T-h}\|\nabla\tau_hv^N(t)\|_p^p+\|\nabla v^N(t)\|_p^p+|\nabla\tau_hv^N(t)\|_p^d+\|\nabla v^N(t)\|_p^ddt\\
 &+C_2\int_0^{T-h}\int_t^{t+h}\|\nabla v^N(s)\|_p^p+\|\nabla v^N(s)\|_p^d+\|v_p^N(s)\|_{W^{1,q}}^qdsdt\\
 &\leq  C_1h\int_0^T \|\nabla v^N(s)\|_p^p+\|\nabla v^N(s)\|_p^ddt+C_2h\int_0^T\|\nabla v^N(s)\|_p^p+\|\nabla v^N(s)\|_p^d+\|v_p^N(s)\|_{W^{1,q}}^qds.\\
 &\leq \tilde {C}h\int_0^T\|\nabla v^N(s)\|_p^p+\|\nabla v^N(s)\|_p^d+\|v_p^N(s)\|_{W^{1,q}}^qds.
\end{split}
 \end{gather}
 With assumption $d\leq p$, $\int_0^T\|\nabla v^N(s)\|_p^d$ is bounded above by
 \[
 \left(\int_0^T\|\nabla v^N(s)\|_p^pds\right)^{\frac{d}{p}}T^{\frac{p-d}{p}}.
 \]
Following Proposition \ref{vbound} and Remark \ref{vpbound}, we conclude
\begin{equation}
\|\tau_hv^N-v^N\|_{L^p(0,T-h;L^p(\Omega;\mathbb{R}^d))}^p\leq Ch,
\end{equation}
which is a bound uniform in $N$. Thus the lemma is proved.
\end{proof}
Next we are going to derive some compactness results from the previous estimates. To reach this goal, we need the help from a variant of the Aubin-Lions Lemma~\cite{chen2012two,simon1986compact}.

The operator $\mathcal{B}$ : $X \rightarrow Y$ is called a (nonlinear) compact operator, if it maps bounded subsets of $X $ to relatively compact subsets of $Y$. Let $L_{\loc}^1(0,T;X) $ be the set of functions $f$ such that for any $0 < t_1 < t_2 < T$ ,  $f\in L^1(t_1,t_2;X)$, equipped with the semi-norms $\|f\|_{L^1(t_1,t_2;X)}$. A subset $F$ of $L_{\loc}^1(0,T;X) $ is called bounded, if for any $0 < t_1 < t_2 < T$, $F$ is bounded in $L^1(t_1,t_2;X)$.

\begin{lemma}\label{Aubin}[Aubin-Lions]
Let $X, Y$ be Banach spaces, $1 \leq p < \infty$ and $\mathcal{B}$ : $X \to Y$ be a (nonlinear) compact operator. Assume that ${F}$ is a bounded subset of $L_{\loc}^1(0,T;X) $ such that ${E} =\mathcal{B}({F}) \subset L^p(0,T;Y)$ and
\begin{itemize}
\item{$E$ is bounded in $L_{\loc}^1(0,T; Y) $,}
\item{$\|\tau_hu - u\|_{L^p(0,T-h;Y)} \to 0$ as $h \to 0+$, uniformly for $u \in E$.} 
\end{itemize} 
Then $E$ is relatively compact in $L^p(0,T;Y)$.
\end{lemma}

Now we are ready to get a candidate of weak solutions through the limit of subsequence of $\{v^N\}_{N\geq 1}$.

\begin{proposition}\label{conv}
Let $v^N$ be the solution to the ODE system (\ref{eq:ode}). There exists a subsequence $\{N_k\}_{k\geq 1}$, $v\in L^{\infty}(0,T;U_p(\Omega)) \cap L^p(0,T; W)$ and a symmetric matrix\\ $\chi \in L^q(0,T;L^q(\Omega;\mathbb{R}^{d\times d}))$, such that as $k\rightarrow \infty$,
\begin{gather}
\begin{split}
&v^{N_{k}} \rightarrow v, \text{  strongly in } L^p(0,T;L^p(\Omega;\mathbb{R}^d)),\\
&v_p^{N_{k}} \rightarrow |v|^{p-2}v=:v_p,\text{  strongly in }L^q(0,T;L^q(\Omega;\mathbb{R}^d)),\\
&\nabla v^{N_{k}} \rightharpoonup \nabla v,\text{  weakly in }L^p(0,T;L^p(\Omega;\mathbb{R}^d)),\\
&|\mathcal{D}(v^{N_{k}})|^{p-2}\mathcal{D}(v^{N_{k}}) \rightharpoonup \chi, \text{  weakly in } L^q(0,T;L^q(\Omega;\mathbb{R}^d)).
\end{split}
\end{gather}
\end{proposition}
\begin{proof}
In Lemma~\ref{Aubin}, take $X=W_0^{1,p}(\Omega;\mathbb{R}^d)$, $Y=L^p(\Omega;\mathbb{R}^d)$, $E=F=\{v^N\}_{N\geq n_*}$ and $\mathcal{B}$ to be the embedding map from $X$ to $Y$. By Proposition \ref{vbound} and the fact that $W_0^{1,p}(\Omega;\mathbb{R}^d)$ is compactly embedded to $L^p(\Omega;\mathbb{R}^d)$, $\{v^N\}_{N\geq n_*}$ is bounded in $L^{\infty}(0,T;L^p(\Omega;\mathbb{R}^d))$, hence $E$ is a bounded set in $L^1(0,T;L^{p}(\Omega;\mathbb{R}^d))$. In addition, by Lemma~\ref{shiftest}, 
\[
\|\tau_h v^N-v^N\|_{L^{p}(0,T-h;L^{p}(\Omega;\mathbb{R}^d))}\to 0
\]
as $h\rightarrow 0+$ uniformly for $N$. Then by Lemma~\ref{Aubin}, $E=\{v^N\}$ is relatively compact in $L^p(0,T;L^p(\Omega;\mathbb{R}^d))$. Hence, there is a subsequence $\{v^{N_k}\}$ and $v\in  L^p(0,T;L^p(\Omega;\mathbb{R}^d))$ such that
\[
v^{N_{k}} \rightarrow v, \text{  strongly in } L^p(0,T;L^p(\Omega;\mathbb{R}^d)).
\]
Since $v^N\in L^\infty(0, T; U_p(\Omega)) $ with  the uniform bound $\|v_0\|_{L^p}$, one has $v\in L^\infty(0, T; U_p(\Omega)) $ with the same bound. 

The strong convergence of $v^{N_k}$ in $L^p(0, T; L^p)$ implies the almost everywhere convergence and thus  
\[
v_p^{N_{k}} \rightarrow |v|^{p-2}v:=v_p, \text{ a.e in } \Omega \times[0,T].
\]
Combining with the fact that
\[
\|v_p^{N_{k}}\|_{L^q(0,T; L^q(\Omega;\mathbb{R}^d))}=\|v^{N_{k}}\|_{L^p(0,T;L^p(\Omega; \mathbb{R}^d))}^{p/q} \rightarrow \|v\|_{L^p(0,T;L^p(\Omega; \mathbb{R}^d))}^{p/q}=\|v_p\|_{L^q(0,T; L^q(\Omega;\mathbb{R}^d))},
\]
one has
\[
v_p^{N_{k}} \rightarrow |v|^{p-2}v:=v_p,\text{  strongly in }L^q(0,T;L^q(\Omega;\mathbb{R}^d)).
\]

From Proposition \ref{vbound}, we know that $\{\nabla v^{N_k}\}$ is bounded in $L^p(0,T; L^p(\Omega; \mathbb{R}^d))$, which is a reflexive space. Then for a subsequence (without relabeling), $\nabla v^{N_k} \rightharpoonup \zeta \in L^p(0,T; L^p(\Omega; \mathbb{R}^d))$. Taking $\phi \in C^\infty(\Omega\times[0,T); \R^{d\times d})$, then
\begin{multline}
\int_0^T\int_{\Omega}\zeta : \phi dxdt=\lim_{k\rightarrow \infty}\int_0^T\int_{\Omega}\nabla v^{N_k}:\phi dxdt\\
=-\lim_{k\rightarrow \infty}\int_0^T\int_{\Omega} v^{N_k}\cdot(\nabla \cdot\phi)dxdt= -\int_0^T\int_{\Omega}v\cdot(\nabla \cdot\phi)dxdt,
\end{multline}
where the last equality follows from the fact that $v^{N_k}\rightarrow v$ in $L^p(0,T;L^p(\Omega;\mathbb{R}^d))$.
Hence, we have
$\nabla v =\zeta$. Then, $v^{N_k}$ converges weakly to $v$ in $L^p(0,T; W^{1,p})$ with $v^{N_k}\in L^p(0,T; W)$. 
So $v\in L^p(0, T; W)$.
Note that
\[
\||\mathcal{D}(v^{N_{k}})|^{p-2}\mathcal{D}(v^{N_{k}})\|_{L^q(0,T;L^q(\Omega;\mathbb{R}^d))}^q=\int_0^T\int_{\mathbb{R}^d}|\mathcal{D}(v^{N_k})|^p\,dxdt < C,
\]
which then yields that $ |\mathcal{D}(v^{N_{k}})|^{p-2}\mathcal{D}(v^{N_{k}}) \rightharpoonup \chi$,  for some $\chi \in L^q(0,T;L^q(\Omega;\mathbb{R}^d))$.
\end{proof}

\section{Existence of weak solutions}\label{5}

To establish the existence of the weak solutions, we need to  identify $\chi$ with $|\mathcal{D}(v)|^{p-2}\mathcal{D}(v)$. Now, define
$G: L^p(0, T; L^p(\Omega; \R^{d\times d}))\to L^q(0, T; L^q(\Omega; \R^{d\times d}))$ by
\begin{gather}
G(\theta):=|\theta|^{p-2}\theta.
\end{gather}
For any $\theta_1, \theta_2 \in L^p(0, T; L^p(\Omega; \R^{d\times d}))$ (see Lemma \ref{gady}, where the product of two matrices is $A:B$),
\begin{equation}
\langle \theta_1-\theta_2,G(\theta_1)-G(\theta_2)\rangle_{L^2_t(0,T;L^2_x(\Omega))}\geq 0,
\end{equation}
which indicates that $G$ is a monotone operator.  We also note that the mapping $\lambda \rightarrow \langle v_2, G(v_1+\lambda v_2)\rangle$ is continuous. 
To verify that $\chi=G(\cD v)$, we will basically apply the following Browder-Minty theorem (\cite[Theorem 10.49]{rr04}).
\begin{lemma}\label{Minty}
Let X  be a real reflexive Banach space.  Let $G$: $X\rightarrow X'$ be a nonlinear, bounded monotone operator satisfying $\forall  v_1,v_2\in X$ the mapping $\lambda \rightarrow \langle v_2,G(v_1+\lambda v_2)\rangle$ is continuous. If $w_n \rightharpoonup w$ in $X$, and $G(w_n)\rightharpoonup \beta$ in $X'$, and
\[
\limsup_{n\rightarrow \infty}~\langle w_n,G(w_n)\rangle\leq\langle w,\beta\rangle,
\]
then $G(w)=\beta$.
\end{lemma}

Below, we will set $X=L^p(0, T; L^q)$
and $X'=L^q(0, T; L^q)$. Moreover, recall that
\begin{equation}
V=\{v\in L^p(0,T;W_0^{1,p}(\Omega,\mathbb{R}^d)),  \nabla\cdot v=0\}.
\end{equation}
We note that $V$ is reflexive since it is a closed subspace of $L^p(0,T;W^{1,p}(\Omega,\mathbb{R}^d))$, which is reflexive. 

To establish the conditions in this lemma, we need to estimate the time regularity of the solutions.

\subsection{Time regularity and finite difference approximation}\label{subsec:timereg}

In the following lemma, we would discuss the time regularity. We aim to figure out the convergence of subsequence $\{\partial_t v_p^{N_k}\}$. 
However, $\partial_t v_p^{N_k}$ is in $W_N'$ which decreases as $N$ becomes large. Hence, we introduce the projection operator $Q_N$ so that $Q_N: W\to W_N$.
Then, we can talk about the convergence of $Q_N^*\partial_tv_p^{N_k}$, where $Q_N^*$ is the conjugate operator of $Q_N$.

Particularly, we will introduce the following. For any $u\in W$, in terms of the  Schauder basis, one has
\begin{gather*}
u=\sum_{k=1}^\infty c_k\phi_k.
\end{gather*}
Define $Q_N: W \to W_N$ to be the projection operator
\begin{gather*}
 Q_N u=\sum_{k=1}^N c_k\phi_k.
\end{gather*}
Similar as in the proof of Proposition \ref{pro:basis}, the Uniform Boundedness principle implies that $Q_N: W\to W_N\subset W$ is uniformly bounded in $N$, i.e.,
\[
\sup\limits_{N\in\mathbb{N}} \|Q_N\|_{W\to W}<\infty.
\]

For a function $v$ in $V$, 
for a.e. $t$, $v(t)\in W$. Hence, $Q_N$ is well-defined on $V$ as well and $Q_N$ is also uniformly bounded in $V$. Let $Q_N^*: V'\longrightarrow V'$ be the conjugate operator of $Q_N$, i.e. for any $u\in V$, $w\in V'$, it holds that
\begin{gather*}
\langle Q_N u,w\rangle_{L^2_t(0,T;L^2_x(\Omega))}=\langle u, Q_N^*w\rangle_{L^2_t(0,T;L^2_x(\Omega))}.
\end{gather*}

We are now ready to show the time regularity results for the sequence $\{Q_N^*\partial_t v_p^{N_k}\}$.
\begin{lemma}\label{pvpconv}
For the subsequence $\{v^{N_k}\}$ in Proposition~\ref{conv},  we can further get a subsequence of $v^{N_k}$ (without relabeling) such that
\begin{equation}
Q_{N_k}^*\partial_t v_p^{N_k} \rightharpoonup  \partial_t v_p, \text{  weakly in } V'.
\end{equation}
Here $V'$ is the dual space of $V$ as in \eqref{spV}. Moreover, for any $w \in V$, we have
\begin{gather}\label{eq:weakderivative}
\langle w,\partial_t\mathcal{P}v_p\rangle_{L^2_t(0,T;L^2_x(\Omega))}-\int_0^T\int_{\Omega}\nabla w:v\otimes v_p\,dxdt +\nu \int_0^T\int_{\Omega}\nabla w: \chi dxdt=0.
\end{gather}
\end{lemma}
\begin{proof}
For any fixed $\varphi\in V$ with $\|\varphi\|_{L^p(0,T;W_0^{1,p}(\Omega, \mathbb{R}^d))}\leq 1$.
Note that 
\[
\langle\varphi, Q_{N_k}^*\partial_t v_p^{N_k}\rangle_{L^2_t(0,T;L^2_x(\Omega))}=\langle Q_{N_k}\varphi,\partial_t v_p^{N_k}\rangle_{L^2_t(0,T;L^2_x(\Omega))}.
\]
Since $Q_{N_k}\varphi(t)\in W_N$ for a.e. $t$, and $Q_{N_k}$ is uniformly bounded, one then has by \eqref{galerkin2} that
\begin{gather*}
\begin{split}
&\langle\varphi, Q_{N_k}^*\partial_t v_p^{N_k}\rangle_{L^2_t(0,T;L^2_x(\Omega))}\\
&=-\langle Q_{N_k}\varphi,v^{N_k}\cdot\nabla v_p^{N_k}\rangle_{L^2_t(0,T;L^2_x(\Omega))}+\nu\langle Q_{N_k}\varphi,\mathcal{L}_p(v^{N_k})\rangle_{L^2_t(0,T;L^2_x(\Omega))}\\
&=\int_0^T\int_{\Omega}\nabla Q_{N_k}\varphi:(v^{N_k}\otimes v_p^{N_k})dxdt-\nu\int_0^T\int_{\Omega}\nabla Q_{N_k}\varphi:\mathcal{D}(v^{N_k})|\mathcal{D}(v^{N_k})|^{p-2}dxdt.
\end{split}
\end{gather*}
Now we estimate the right hand side term by term. For the first term, By H\"older inequality,
\begin{gather}\label{eq:transport}
\left|\int_0^T\int_{\Omega}\nabla Q_{N_k}\varphi:(v^{N_k}\otimes v_p^{N_k})dxdt\right|\leq \int_0^T\int_{\Omega}\frac{|\nabla Q_{N_k}\varphi|^p}{p}+\frac{|v^{N_k}|^{2p}}{2p}+\frac{2p-3}{2p}|v_p^{N_k}|^{\frac{2p}{2p-3}}dxdt.
\end{gather}
Since $Q_{N_k}$ is bounded in $L^p(0, T; W_0^{1,p})$, $\int_0^T\int_{\Omega}|\nabla Q_{N_k}\varphi|^p\,dxdt$ is bounded uniformly in $N_k$. Applying Gagliardo-Nirenberg inequality, one has
\begin{align}\label{e:gagv}
\|v^{N_k}\|_{2p}^{2p}\leq C\|\nabla v^{N_k}\|_p^d\|v^{N_k}\|_p^{2p-d},
\end{align}
and
\begin{align}\label{e:gagvp}
\|v_p^{N_k}\|_{\frac{2p}{2p-3}}^{2p/(2p-3)}\leq C\|\nabla v_p^{N_k}\|_q^{d/(2p-3)}\|v_p^{N_k}\|_q^{(2p-d)/(2p-3)}.
\end{align}
Using the estimates in Proposition \ref{vbound}, and the fact that $\frac{d}{2p-3}\leq q$ and $\frac{2p-d}{2p-3}\leq q$ for $p\geq d\geq 2$, one concludes the boundedness.

For the second term, it is easy to check the boundedness, since H\"older inequality yields that
\[
\left|\int_0^T\int_{\Omega}\nabla Q_{N_k}\varphi:\mathcal{D}(v^{N_k})|\mathcal{D}(v^{N_k})|^{p-2}dxdt\right|\leq \int_0^T\int_{\Omega} \frac{|\nabla Q_{N_k}\varphi|^p}{p}+\frac{|\mathcal{D}(v^{N_k})|^p}{q}dxdt.
\]
Again, using the estimates in Proposition \ref{vbound}, one gets the boundedness.

Therefore, $Q_{N_k}^*\partial_t v_p^{N_k}$ is bounded in $V'$. Since $V$ is reflexive, then there is a subsequence (without relabeling) and $\alpha \in V'$ such that
\begin{gather}
Q_{N_k}^*\partial_t v_p^{N_k} \rightharpoonup  \alpha \text{ weakly in } V'.
\end{gather}
For any $\psi \in C_c^{1}(\Omega\times [0, T)])$ with $\nabla\cdot \psi=0$, one has
\begin{gather*}
\begin{split}
\langle\psi,Q_{N_k}^*\partial_t v_p^{N_k}\rangle_{L^2_t(0,T;L^2_x(\Omega))}&=\langle Q_{N_k}\psi,\partial_t v_p^{N_k}\rangle_{L^2_t(0,T;L^2_x(\Omega))}\\
&=-\langle\partial_tQ_{N_k}\psi,  v_p^{N_k}\rangle_{L^2_t(0,T;L^2_x(\Omega))}-\int_{\Omega}Q_{N_k}\psi(x,0) v_p^{N_k}(x,0)dx\\&\to -\langle\partial_t\psi,  v_p\rangle_{L^2_t(0,T;L^2_x(\Omega))}-\int_{\Omega}\psi(x,0) v_p(x,0)dx,
\end{split}
\end{gather*}
as $k\to \infty$. Note that the completion of $C_c^1(\Omega\times [0, T)])$ with zero divergence in $W_0^{1,p}$ is $V$. Hence $\alpha=\partial_t v_p$ in $V'$.

Now for any $\varphi\in C_c^{1}( \Omega\times [0, T))$ with $\nabla\cdot \varphi=0$, by the convergence in Proposition~\ref{conv} one clearly has $v^{N_k}\otimes v_p^{N_k}\to v\otimes v_p$ strongly in $L^1(0,T; L^1(\Omega;\R^{d\times d}))$ and hence it holds that
\[
\l \varphi, \partial_t  v_p\r_{L^2_t(0,T;L^2_x(\Omega))}
-\int_0^T\int_{\Omega}\nabla\varphi : v\otimes v_p\,dxdt+\nu\int_0^T\int_{\Omega}\nabla\varphi:\chi\,dxdt=0.
\]
Similar to the estimate in \eqref{eq:transport}, $v\otimes v_p\in L^q(0, T; L^q(\Omega;\R^{d\times d}))$. Hence, by density argument, we can replace $\varphi$ by any $w\in V$.
\end{proof}

\begin{remark}
We do not have the convergence $v^N\otimes v_p^N\to v\otimes v_p$ in $L^q(0, T; L^q(\Omega;\R^{d\times d}))$.
\end{remark}

Next, we need a technical result to obtain the chain rule for the weak time derivative of the $L^p$ integral for $v$. In the proof of the original work \cite{liliu17}, there is a gap to establish the chain rule (specifically, in the paragraph and equation below Equation (97)). The method here can be used to fill that gap there. 
Define the energy function
\begin{gather}
H(t):=\frac{1}{q}\int |v(x,t)|^p\,dx
=\frac{1}{q}\int |v_p(x,t)|^q\,dx.
\end{gather}
Consider the finite time differences
\begin{gather}
D_h^+g(t):=\frac{1}{h}(\tau_h g(t)-g(t))
\end{gather}
and
\begin{gather}
D_h^-g(t):=\frac{1}{h}(g(t)-\tau_{-h}g(t)).
\end{gather}
We have
\begin{proposition}\label{timedev}
The time differences
$D_h^+  v_p(t)1_{[0, T-h]}(t)$
and $D_h^-  v_p(t)1_{[h, T]}(t)$ are bounded uniformly in $V'$  and both have subsequences (without labelling) converge weakly to $\partial_t  v_p$ as $h\to 0$ in $V'$.
Moreover, there is a version of the mapping $t\mapsto H(t)$ that is continuous with $H(0)=\frac{1}{q}\|v_0\|_{L^p}^p$ and satisfies for any $0\le s\le t\le T$:
\begin{equation}
   \int_s^t\int_{\Omega} v(\tau)\partial_t  v_p(\tau)\,dxd\tau=H(t)-H(s).
\end{equation}
\end{proposition}

\begin{proof}
Take $\varphi\in V$, $T>h>0$. Let $Q_N^*$ be the conjugate operator of $Q_N$ as defined before. Then by \eqref{galerkin2}, one has
\begin{multline}
\int_0^{T}\int_{\Omega} \varphi(t)\cdot Q^*_N D_h^-  v_p\,dx dt=\int_h^T\frac{1}{h}\int_{t-h}^{t}\int_{\Omega} \nabla_x Q_N \varphi(t): v^N(\tau) \otimes v_p^N(\tau) dx d\tau dt \\
-\nu\int_h^T\frac{1}{h}\int_{t-h}^t\int_{\Omega}\nabla_x Q_N  \varphi(t): \mathcal{D}(v^N(\tau))|\mathcal{D}v^N(\tau)|^{p-2}dxd\tau dt
=:I_1+I_2.
\end{multline}
Next, we estimate these two terms. By Young's inequality,
\begin{equation}
\begin{split}
|I_1| &\leq \int_h^T\int_{\Omega}\frac{|\nabla_x Q_N \varphi(t)|^p}{p}dt+\int_{h}^T\frac{1}{h}\int_{t-h}^t\int_{\Omega}\frac{|v^N(\tau) \otimes v_p^N(\tau)|^q}{q}\, dxd\tau dt\\
& \leq \frac{C_p}{p}\|\varphi\|_{L^p(0;T;W^{1,p}_0)}^p+\int_0^{T}\int_{\Omega}\frac{|v^N(s) \otimes v_p^N(s)|^q}{q}\, ds,
\end{split}
\end{equation}
which is bounded as in \eqref{eq:transport}, \eqref{e:gagv} and \eqref{e:gagvp}. In addition, one similarly has
\begin{equation}
\begin{split}
    |I_2| &\leq \int_h^T\int_{\Omega}\frac{|\nabla_x Q_N \varphi(t)|^p}{p}dt+\int_{h}^T\frac{1}{h}\int_{t-h}^t\int_{\Omega}\frac{\left(|\mathcal{D}v^N(\tau)|^{p-1}\right)^q}{q}\, dxd\tau dt\\
    &\leq \frac{C_p}{p}\|\varphi\|_{L^p(0;T;W^{1,p}_0)}^p+\int_0^{T}\int_{\Omega}\frac{\left|\mathcal{D}v^N(s)\right|^{(p-1)q}}{q}\, dxds,
\end{split}
\end{equation}
which is also bounded due to the fact that $\mathcal{D}(v^N)$ is bounded in $L^p(0,T;L^p(\Omega))$. Hence, $ \frac{1}{h}(Q^*_N\mathcal{P}v_p^N(t)-Q^*_N\mathcal{P}v_p^N(t-h))1_{[h, T]}(t)$ is uniformly (in $N$ and $h$) bounded in $V'$. Letting $N\to \infty$, by the strong convergence of $v_p^N$ to $v_p$, we have that $D_h^- v_p(t) 1_{[h,T]}(t)$ is uniformly (in $h$) bounded in  $V'$.
Similar arguments hold for $D_h^+  v_p(t)1_{[0, T-h]}$.

Now, up to subsequence, as $h\to 0$, $D_h^-  v_p(t) 1_{[h,T]}(t)$ would have a weak limit in $V'$, denoted by $\gamma$. By pairing with a smooth function,
it is not hard to identify that $\gamma$ is just $\partial_t v_p$. Similarly, $D_h^+  v_p(t)1_{[0, T-h]}$ would have a subsequence converging to $\partial_t v_p$ as well.

Now, since $v\in V$, let $s,t\in (h,T)$, we have
\begin{equation}
\begin{split}
\int_s^t \int_{\Omega}v(\tau) D_h^- v_p(\tau) \, dxd\tau&=\frac{1}{h}\int_s^t\int_{\Omega} |v(\tau)|^p-v(\tau)v_p(\tau-h)\, dxd\tau\\
&\geq \frac{1}{h}\int_s^t\int_{\Omega}  |v(\tau)|^p-\frac{|v(\tau)|^p}{p}-\frac{1}{q}|v_p(\tau-h)|^q\, dxd\tau\\
&= \frac{1}{qh}\int_t^{t-h}\int_{\Omega}|v(\tau)|^p\, dxd\tau-\frac{1}{qh}\int_s^{s-h}\int_{\Omega}|v(\tau)|^p\, dxd\tau.
\end{split}
\end{equation}
Then, as $h\to 0^+$, the left hand side converges to $\int_s^t\int_{\Omega}v(\tau)\partial_t  v_p(\tau)\,dxd\tau$ due to the weak convergence of $D_h^-  v_p 1_{[h, T]}$ and thus $D_h^-  v_p 1_{[s, t]}$ in $V'$. The right hand side tends to $\frac{1}{q}\|v(t)\|^p_{L^p}-\frac{1}{q}\|v(s)\|^p_{L^p}$ as $h$ tends to zero, for almost every $t,s\in (0,T)$. Hence, for almost every $s<t, s, t\in (0, T)$ one has
\[
\int_s^t\int_{\Omega}v(\tau)\partial_t  v_p(\tau)\,dxd\tau \ge H(t)-H(s).
\]
Moreover, one may do the same thing for $D_h^+  v_p(t)$ to get for almost every $s, t\in (0, T-h)$ that
\[
\int_s^t\int_{\Omega}v(\tau) D_h^+  v_p(\tau)\,dx d\tau
\le \frac{1}{qh}\int_t^{t+h}\int_{\Omega}|v(\tau)|^p\, dxd\tau-\frac{1}{qh}\int_s^{s+h}\int_{\Omega}|v(\tau)|^p\, dxd\tau.
\]
By the same argument, one has for 
almost every $s<t, s, t\in (0, T)$ one has
\[
\int_s^t\int_{\Omega}v(\tau)\partial_t  v_p(\tau)\,dxd\tau \le H(t)-H(s).
\]
Hence, we have the chain to hold
\begin{equation}\label{Hdev}
\begin{split}
H(t)-H(s)&= \int_s^t\int_{\Omega} v(\tau)\partial_t   v_p(\tau)\,dxd\tau.
\end{split}
\end{equation}
for almost every $s, t\in (0, T)$ with $s<t$. Since the right hand side is continuous in $s, t$ so $H$ can be made into a continuous function.

Moreover, to see $H(0^+)=\frac{1}{q}\|v_0\|_{L^p}^p$, we note that
\[
\left|\frac{1}{q}\int_{\Omega}|v^N(t)|^p\,dx
-\frac{1}{q}\int_{\Omega}|v^N(0)|^p\,dx\right|
\le \left|\int_0^t\int_{\Omega} v^N(\tau)\partial_t v_p^N(\tau)\,d\tau \right| \le Ct,
\]
where $C$ is uniform in $N$. As $N\to \infty$, 
$\frac{1}{q}\int_{\Omega}|v^N(0)|^p\,dx\to \frac{1}{q}\|v_0\|_{L^p}^p$ and for almost every $t$, $\frac{1}{q}\int_{\Omega}|v^N(t)|^p\,dx$ converges to $H(t)$. Hence, $H(0+)$ is given as mentioned. This also means in \eqref{Hdev} we can take $t=T$ and $s=0$.
\end{proof}

\subsection{Existence of weak solutions and the energy dissipation equality}

In this subsection, we first identify $\chi$ and then prove the existence of weak solutions. 
\begin{lemma}\label{intconv}
In $L^q(0, T; L^q(\Omega; \R^{d\times d}))$, we have
\begin{gather}
\chi=|\cD v|^{p-2}\cD v.
\end{gather}
In other words, for any $\varphi\in C_c^\infty([0,T)\times\Omega)$,
\begin{gather}
\int_0^T\int_\Omega\nabla\varphi:\chi dxdt=\int_0^T\int_\Omega\nabla\varphi:\mathcal{D}(v)|\mathcal{D}(v)|^{p-2}dxdt.
\end{gather}
\end{lemma}

\begin{proof}	
Taking $w=v $ in \eqref{eq:weakderivative}, due to the fact the $\chi$ is symmetric, one has
\begin{gather*}
	\langle v, \partial_t v_p\rangle_{L^2_t(0,T;L^2_x(\Omega))}=-\nu\int_0^T\int_\Omega\nabla v:\chi dxdt
 =-\nu\int_0^T\int_\Omega\cD(v):\chi dxdt.
\end{gather*}
In fact, we need to show
\[
\int_0^T\int_{\Omega}\nabla v:v\otimes v_pdxdt=0
\]
To justify this, first recalling~\eqref{e:gagv} and~\eqref{e:gagvp} one has $v\in L^{2p}(0, T; L^{2p}(\Omega))$ and $v_p\in L^{r}(0, T; L^r(\Omega))$ with $r=\frac{2p}{2p-3}$. We can extend $v$ to be defined in $\R^d$ such that $v=0$ for $x\notin\Omega$. Then, $v\in L^{2p}(0, T; L^{2p}(\R^d))$ and $\nabla v\in L^p(0, T; L^p(\R^d))$. Then, $v_{\e}:=v*J_{\e}$ for a mollifier $J_{\e}$. Then, one has $v_{\e}\to v$ in $L^{2p}(0, T; L^{2p}(\R^d))$, $\nabla v_{\e}\to \nabla v$ in $L^p(0, T; L^p(\R^d))$ and $(v_{\e})_p\to v_p$ in $L^{r}(0, T; L^r(\Omega))$. Hence, due to $\nabla\cdot v_{\e}=0$, one has
\[
0=\int_0^T\int_{\Omega}\nabla v_{\e}:v_{\e}\otimes (v_{\e})_pdxdt
\to \int_0^T\int_{\Omega}\nabla v:v\otimes v_pdxdt.
\]

On the other hand, with the same notation as in Lemma \ref{pvpconv} and Proposition \ref{timedev},
\begin{gather}\label{eq:aux0}
\frac{\|v^{N_k}(t)\|^p_{L^p}}{q}-\frac{\|v^{N_k}(s)\|^p_{L^p}}{q}
\to \frac{\|v(t)\|^p_{L^p}}{q}-\frac{\|v(s)\|^p_{L^p}}{q}=\int_s^t \int_{\Omega}v(\tau) \partial_t  v_p(\tau)\, dxd\tau.
\end{gather}
By the continuity argument same as in the end of the proof on Proposition \ref{timedev}, we can take $t=T$ and $s=0$. Besides, the left hand side of \eqref{eq:aux0} equals
\begin{gather}\label{eq:aux1}
\langle v^{N_k}, Q_N^*\partial_t v_p^{N_k}\rangle_{L^2_t(0,T;L^2_x(\Omega))}
=\langle v^{N_k}, \partial_t v_p^{N_k}\rangle_{L^2_t(0,T;L^2_x(\Omega))}
=-\int_0^T\int_{\Omega}|\cD{v^{N_k}}|^pdxdt.
\end{gather}
Hence, one actually has,
\[
\lim_{k\rightarrow\infty}\int_0^T\int_{\Omega}|\cD{v^{N_k}}|^pdxdt =\int_0^T\int_{\Omega} \cD(v):\chi dxdt.
\]

By the property of $G$ and Lemma \ref{Minty}, one has
\[
\chi=G(\cD(v)).
\]
\end{proof}

We get the existence of weak solution to $p$-Navier-Stokes equations.
\begin{theorem}\label{mainthm}
Let $\Omega$ be a bounded domain in $\mathbb{R}^d$ with $C^{\infty}$ boundary and $p\geq d\geq2$. Let $v_0\in U_p(\Omega)$.  There exists a global weak solution to  initial/boundary value problem of the $p$-Navier-Stokes equations (Equations (\ref{pNSe})) in the sense of Definition \ref{def:weak}. 
\end{theorem}
\begin{proof}
For any $\psi\in C_c^{\infty}(\Omega\times[0,T))$, one has
\[
\int_0^T\int_{\Omega}\nabla\psi\cdot v^{N_k} dxdt=0,
\]
since $v^{N_k}$ is divergence free and disappears on the boundary. Using Lemma \ref{conv}, sending $k$ to $\infty$, one has
\[
\int_0^T\int_{\Omega}\nabla\psi\cdot v dxdt=0.
\]
 For any $\varphi\in C_c^{\infty}(\Omega\times [0,T),\mathbb{R}^d)$, one has
\begin{multline*}
 \int_0^T\int_{\Omega}v_p^{N_k}\cdot \partial_t \varphi dxdt+\int_0^T\int_{\Omega}\nabla\varphi : (v^{N_k}\otimes v_p^{N_k})dxdt \\
 -\nu\int_0^T\int_{\Omega}\nabla\varphi : \mathcal{D}(v^{N_k})|\mathcal{D}(v^{N_k})|^{p-2}dxdt +\int_{\Omega}|v_0^{N_k}|^{p-2}v_0^{N_k}\cdot\varphi(x,0)dx=0.
\end{multline*}
 Again using Lemma \ref{conv}, sending $k$ to $\infty$, one gets
 \begin{multline*}
\int_0^T\int_{\Omega}v_p\cdot \partial_t \varphi dxdt+\int_0^T\int_{\Omega}\nabla\varphi : (v\otimes v_p)dxdt \\
-\nu\int_0^T\int_{\Omega}\nabla\varphi : \mathcal{D}(v)|\mathcal{D}(v)|^{p-2}dxdt+\int_{\Omega}|v_0|^{p-2}v_0\cdot\varphi(x,0)dx=0.
\end{multline*}
For the time regularity, note that for some $C$ independent of $N_k$,
\[
\|\tau_hv^{N_k}-v^{N_k}\|_{L^p(0,T-h;L^p(\Omega;\mathbb{R}^d))}^p\leq Ch,
\] 
by Lemma \ref{shiftest}. Using the fact that $v^{N_{k}} \rightarrow v$ strongly in $L^p(0,T;L^p(\Omega;\mathbb{R}^d))$, one then has
\[
\|\tau_hv-v\|_{L^p(0,T-h;L^p(\Omega;\mathbb{R}^d))}^p\leq Ch.
\]
Thus the time regularity is proved. From all above, we proved that $v$ is a weak solution to the $p$-Navier-Stokes problem for given $T$.

Note that $T$ is arbitrary, one may use a diagonal argument to extract a subsequence of $v^{N_k}$ that converges in the sense listed in Proposition \ref{conv} for every $[0, n]$. The limit of this subsequence is then a weak solution on any bounded interval and thus a global weak solution.
\end{proof}

Moreover, we have the following the energy dissipation law for the weak solutions, which follows directly from Proposition \ref{timedev} and the argument in the proof of Lemma \ref{intconv}.
\begin{proposition}
It holds that
\begin{gather}
H(t)-H(s)=-\nu\int_s^t \int_{\Omega}|\cD(v)|^p\,dxd\tau.
\end{gather}
\end{proposition}

\section*{Acknowledgements}

This work was financially supported by the National Key R\&D Program of China, Project Number 2020YFA0712000 and  2021YFA1001200. The work of Y. Feng was partially supported by Science and Technology Commission of Shanghai Municipality (No. 22DZ2229014). The work of L. Li was partially supported by Shanghai Science and Technology Commission (Grant No. 21JC1403700, 20JC144100), the Strategic Priority Research Program of Chinese Academy of Sciences, Grant No. XDA25010403 and NSFC 12031013. The work of J.-G. Liu was partially supported by NSF DMS-2106988.  The work of X. Xu was partially supported by the NSFC 12101278, and Kunshan Shuangchuang Talent Program kssc202102066.

\appendix

\section{Equivalence of the norms}\label{app:equivnorm}
In this section, we show the following equivalency of the $W^{1,p}$ norm:
\begin{equation}\label{eq:Wpnorm}
C'(\|u\|_{L^p}+\|\mathcal{D}(u)\|_{L^p}) \leq \|u\|_{L^p}+\|D u\|_{L^p}\leq C(\|u\|_{L^p}+\|\mathcal{D}(u)\|_{L^p}),
\end{equation}
where $C, C'$ are constants, $\mathcal{D}(u)=\frac{1}{2}(\nabla u+\nabla u^{T})$. The proof can be found on \cite{carlos1992}, Proposition 1.1, for completion, we sketch the main steps here.

The first inequality of \eqref{eq:Wpnorm} is trivial. For the second inequality, observe that one has the following relations between $u$ and $\mathcal{D}(u)$:
\[
\mathcal{D}(u)_{ij}=\frac{\partial u_i}{\partial x_j}+\frac{\partial u_j}{\partial x_i},
\]
\[
\frac{\partial^2 u_i}{\partial x_j\partial x_k}=\frac{\partial \mathcal{D}(u)_{ik}}{\partial x_j}+\frac{\partial\mathcal{D}(u)_{ij}}{\partial x_k}-\frac{\partial\mathcal{D}(u)_{jk}}{\partial x_i}.
\]
Hence, suppose $\|u\|_{L^p}+\|\mathcal{D}(u)\|_{L^p}$ is finite, then we have $\frac{\partial^2 u_i}{\partial x_j\partial x_k}\in W^{-1,p}$. By Theorem 1.1 in \cite{carlos1992}, 
\[
\|\partial_j u_i\|_{L^p}\le C(\|\partial_j u_i\|_{W^{-1,p}}+\|\nabla \partial_ju_i\|_{W^{-1, p}})
\]
so that $\partial_j u_i\in L^p$. By Open Mapping Theorem, these two norms are equivalent.

\bibliographystyle{plain}
\bibliography{pEuler_Galerkin_Symmetric}

\begin{thebibliography}{10}

\bibitem{bblWebsite}
\url{https://www.claymath.org/millennium-problems}.

\bibitem{arnold1966geometrie}
Vladimir Arnold.
\newblock Sur la g{\'e}om{\'e}trie diff{\'e}rentielle des groupes de lie de
  dimension infinie et ses applications {\`a} l'hydrodynamique des fluides
  parfaits.
\newblock In {\em Annales de l'institut Fourier}, volume~16, pages 319--361,
  1966.

\bibitem{arnold2008topological}
Vladimir~I Arnold and Boris~A Khesin.
\newblock {\em Topological methods in hydrodynamics}, volume 125.
\newblock Springer Science \& Business Media, 2008.

\bibitem{bellout95}
H.~Bellout.
\newblock On a special {S}chauder basis for the {S}obolev spaces
  ${W}_0^{1,p}({\Omega}) $.
\newblock {\em Illinois Journal of Mathematics}, 39(2):187--195, 1995.

\bibitem{breit2017existence}
Dominic Breit.
\newblock {\em Existence theory for generalized Newtonian fluids}.
\newblock Academic Press, 2017.

\bibitem{chen2012two}
Xiuqing Chen and Jian-Guo Liu.
\newblock Two nonlinear compactness theorems in {$L^p (0, T; B)$}.
\newblock {\em Applied Mathematics Letters}, 25(12):2252--2257, 2012.

\bibitem{crochet2012numerical}
Marcel~J Crochet, Arthur~Russell Davies, and Kenneth Walters.
\newblock {\em Numerical simulation of non-Newtonian flow}.
\newblock Elsevier, 2012.

\bibitem{dl1998}
L.~Damascelli.
\newblock Comparison theorems for some quasilinear degenerate elliptic
  operators and applications to symmetry and monotonicity results.
\newblock In {\em Annales de l'Institut Henri Poincare (C) Non Linear
  Analysis}, volume~15, pages 493--516. Elsevier, 1998.

\bibitem{evans2022partial}
Lawrence~C Evans.
\newblock {\em Partial differential equations}, volume~19.
\newblock American Mathematical Society, 2022.

\bibitem{fuvcik1972existence}
Svatopluk Fu{\v{c}}{\'\i}k, Old{\v{r}}ich John, and Jind{\v{r}}ich Ne{\v{c}}as.
\newblock On the existence of schauder bases in sobolev spaces.
\newblock {\em Commentationes Mathematicae Universitatis Carolinae},
  13(1):163--175, 1972.

\bibitem{galdi2011}
G.~P Galdi.
\newblock {\em An introduction to the mathematical theory of the
  {N}avier-{S}tokes equations: Steady-state problems}.
\newblock Springer Science \& Business Media, 2011.

\bibitem{liliu17}
L.~Li and J.-G. Liu.
\newblock $p$-{Euler} equations and $ p $-{Navier}-{Stokes} equations.
\newblock {\em J. Differ. Equations}, 264(7), 2018.

\bibitem{lions1996mathematical}
Pierre-Louis Lions.
\newblock {\em Mathematical Topics in Fluid Mechanics: Volume 2: Compressible
  Models}, volume~2.
\newblock Oxford University Press on Demand, 1996.

\bibitem{liu2021existence}
Jian-Guo Liu and Zhaoyun Zhang.
\newblock Existence of global weak solutions of $ p $-navier-stokes equations.
\newblock {\em Discrete and Continuous Dynamical Systems-B}, 27(1):469--486,
  2021.

\bibitem{carlos1992}
C.~Par\'es.
\newblock Existence, uniqueness and regularity of solution of the equations of
  a turbulence model for incompressible fluids.
\newblock {\em Applicable Analysis}, 43(3-4):245--296, 1992.

\bibitem{rr04}
M.~Renardy and R.~C Rogers.
\newblock {\em An introduction to partial differential equations}, volume~13.
\newblock Springer Science \& Business Media, 2006.

\bibitem{simon1986compact}
Jacques Simon.
\newblock Compact sets in the space {$L^p (0, T; B)$}.
\newblock {\em Annali di Matematica pura ed applicata}, 146:65--96, 1986.

\bibitem{tartar2006introduction}
Luc Tartar.
\newblock {\em An introduction to Navier-Stokes equation and oceanography},
  volume~1.
\newblock Springer, 2006.

\bibitem{vazquez2007porous}
Juan~Luis V{\'a}zquez.
\newblock {\em The porous medium equation: mathematical theory}.
\newblock Oxford University Press on Demand, 2007.

\bibitem{vdbh2009}
D.~Veiga and H.~Beirao.
\newblock {Navier}--{Stokes} equations with shear thinning viscosity regularity
  up to the boundary.
\newblock {\em Journal of Mathematical Fluid Mechanics}, 11(2):258--273, 2009.

\end{thebibliography}
\end{document}